\documentclass[11pt]{article} 

\usepackage[utf8]{inputenc} 

\usepackage{geometry} 
\usepackage{graphicx} 
\usepackage{color}

\definecolor{Red}{rgb}{1,0,0}
\definecolor{Blue}{rgb}{0,0,1}
\definecolor{Olive}{rgb}{0.41,0.55,0.13}
\definecolor{Green}{rgb}{0,1,0}
\definecolor{MGreen}{rgb}{0,0.8,0}
\definecolor{DGreen}{rgb}{0,0.55,0}
\definecolor{Yellow}{rgb}{1,1,0}
\definecolor{Cyan}{rgb}{0,1,1}
\definecolor{Magenta}{rgb}{1,0,1}
\definecolor{Orange}{rgb}{1,.5,0}
\definecolor{Violet}{rgb}{.5,0,.5}
\definecolor{Purple}{rgb}{.75,0,.25}
\definecolor{Brown}{rgb}{.75,.5,.25}
\definecolor{Grey}{rgb}{.5,.5,.5}

\usepackage{booktabs} 
\usepackage{array} 
\usepackage{paralist} 
\usepackage{verbatim} 
\usepackage{subfigure} 
\usepackage{amsmath,amssymb,amsthm,mathrsfs}
\usepackage[colorlinks]{hyperref}
\usepackage{blkarray}

\usepackage{fancyhdr} 
\makeatletter
\newtheorem*{rep@theorem}{\rep@title}
\newcommand{\newreptheorem}[2]{%
\newenvironment{rep#1}[1]{%
 \def\rep@title{#2 \ref{##1}}%
 \begin{rep@theorem}}%
 {\end{rep@theorem}}}
\makeatother

\theoremstyle{plain}

\newtheorem{lemma}{Lemma}[section]

\newtheorem{theorem*}{Theorem}   
\newtheorem{lemma*}{Lemma} 
\newtheorem{corollary*}{Corollary} 
\newtheorem*{remark*}{Remark}
\newtheorem*{example*}{Example}

\newtheorem{remark}{Remark}[section]

\newlength{\widebarargwidth}
\newlength{\widebarargheight}
\newlength{\widebarargdepth}

\theoremstyle{definition}
\newtheorem{definition}{Definition}[section]

\def\cC{{\cal C}}
\def\cD{{\cal D}}

\def\cF{{\cal F}}
\def\cG{{\cal G}}

\def\cK{{\cal K}}

\def\cN{{\cal N}}

\allowdisplaybreaks 

\begin{document}

\begin{center}

{\bf{\Large{Persistence of centrality in random growing trees}}}

\vspace*{.25in}

\begin{tabular}{ccc}
{\large{Varun Jog}} & \hspace*{.75in} & {\large{Po-Ling Loh}} \\ 
{\large{\texttt{varunjog@wharton.upenn.edu}}} & & {\large{\texttt{loh@wharton.upenn.edu}}} \vspace{.2in}
 \\
Departments of Statistics \& CIS & \hspace{.2in} & Department of Statistics \\
Warren Center for Network and Data Sciences && The Wharton School \\
University of Pennsylvania && University of Pennsylvania \\ Philadelphia, PA 19104 & & Philadelphia, PA 19104
\end{tabular}

\vspace*{.2in}

November 2015

\vspace*{.2in}

\end{center}

\begin{abstract}
We investigate properties of node centrality in random growing tree models. We focus on a measure of centrality that computes the maximum subtree size of the tree rooted at each node, with the most central node being the tree centroid. For random trees grown according to a preferential attachment model, a uniform attachment model, or a diffusion processes over a regular tree, we prove that a single node persists as the tree centroid after a finite number of steps, with probability 1. Furthermore, this persistence property generalizes to the top $K \ge 1$ nodes with respect to the same centrality measure. We also establish necessary and sufficient conditions for the size of an initial seed graph required to ensure persistence of a particular node with probability $1-\epsilon$, as a function of $\epsilon$: In the case of preferential and uniform attachment models, we derive bounds for the size of an initial hub constructed around the special node. In the case of a diffusion process over a regular tree, we derive bounds for the radius of an initial ball centered around the special node. Our necessary and sufficient conditions match up to constant factors for preferential attachment and diffusion tree models.
\end{abstract}

\section{Introduction}

Heterogeneity is a common phenomenon arising naturally in many network datasets. Although some networks exist in which connections form at random between approximately exchangeable individuals, it is usually more realistic to assume that certain nodes occupy a more favorable position in the network than others. This could arise because particular nodes possess attributes that increase their likelihood of connectivity in relation to other nodes in the network. It could also be due to strategic network formation, which---even when nodes are indistinguishable---may settle on an equilibrium position where one node is in a more powerful position than the others, due to tradeoffs between the cost and utility of maintaining pairwise connections~\cite{JacWol96}. A third possibility is that the network is formed over a period of time, and older nodes are more likely to possess a higher degree of connectivity than newer nodes in the network. In order to quantify the amount of heterogeneity present in a network, various summary statistics have been proposed, including degree distributions, average path lengths, clustering coefficients, and different measures of centrality~\cite{Jack08}.

The Bar\'abasi-Albert model, also known as the preferential attachment model, is one popular probabilistic framework for modeling the dynamics of a random growing graph~\cite{BarAlb99}. In this model, each new node connects to existing nodes with probability proportional to the degrees of the nodes in the previous time step. Galashin~\cite{Gal14} recently showed that with probability 1, a single node emerges as a persistent hub in a preferential attachment network, meaning it remains the highest-degree node in the graph after a finite number of time steps. In contrast, such a phenomenon does not occur for the uniform attachment model, in which each new node connects to existing nodes uniformly at random. Intuitively, this is due to the fact that newly created nodes have a relatively high probability of replacing the current node of highest degree after the graph evolves further. Although older nodes in a uniform attachment model may not have a substantial lead in terms of degree, it is nonetheless reasonable to expect older nodes in the network to exhibit a higher level of connectivity according to some suitable measure. We confirm this intuition by tracking the dynamics of a different summary statistic, the \emph{centrality} of a node in a random growing network, and prove that a single persistent node of highest centrality emerges almost surely in the case of uniform attachment trees, preferential attachment trees, and another related random growing tree arising from a diffusion process over a $d$-regular tree.

Numerous notions of centrality have been introduced in the literature on social networks, including degree centrality (also known as the maximum degree), distance centrality, betweenness centrality, and eigenvalue centrality (see, e.g., \cite{Bor05, Jack08}). In the case of trees, many popular notions of centrality conveniently coincide in terms of the most central node, which we will refer to as the tree centroid. The notion of a tree centroid was first introduced by Jordan \cite{Jor1869}, where it was originally called the branch weight centroid, and was subsequently studied by various authors using equivalent characterizations such as the distance center \cite{Jack08}, rumor center \cite{ShahEtAl11}, median vertex of a graph \cite{Zel68}, security center, \cite{Sla75}, accretion center \cite{Sla81}, and telephone center \cite{Mit78}. In a random growing tree, we will call a vertex a ``persistent centroid" if it is the tree centroid for all but finite moments in time. Our first main contribution therefore establishes the existence of a persistent centroid in each of the random growth models described above.

Centrality measures in random graphs have also been analyzed recently in the probability theory literature for devising root-finding algorithms in growing networks~\cite{BubEtal14, KhiLoh15}. In such settings, selecting the top $K$ nodes with respect to an appropriate centrality measure yields a confidence set for the initial node in the random graph, where $K$ is only required to be a function of the error probability, and not the total number of nodes. Motivated by these findings, our next contribution is to generalize the result on the persistence of a single centroid to the case of the top $K$ central nodes. Consequently, the confidence set generated by a root-finding algorithm based on this measure of centrality is guaranteed to stabilize after a finite amount of time, which is a desirable property from the point of view of robustness.

As a final contribution, we address the following natural question: Suppose an individual wants to ensure that he or she is the persistent centroid of the network. The individual may boost his or her probability of becoming the persistent centroid by creating a large number of initial links to other nodes (i.e., forming a large ``seed hub," for which it is the center node). How large should the initial hub be in order to ensure that the individual becomes the unique persistent centroid, with probability $1-\epsilon$? We answer this question for each of the random growing tree models. In the case of preferential and uniform attachment, we establish necessary and sufficient conditions for the initial hub size $k$. In a $d$-regular tree, we instead surround the first individual with a seed graph consisting of all nodes within radius $r$ of that node, and derive necessary and sufficient conditions for the size of $r$. Our results are summarized in the following table:
\begin{center}
\begin{tabular}{|c||c|c|}
\hline
\textbf{model} & \textbf{necessary condition} & \textbf{sufficient condition} \\
\hline \hline
preferential attachment & $k \ge c \log(1/\epsilon)$ & $k = C \log(1/\epsilon)$ \\
\hline
uniform attachment & $k \ge c' \frac{\log (1/\epsilon)}{\log \log(1/\epsilon)}$ & $k = C' \log(1/\epsilon)$ \\
\hline
$d$-regular diffusion & $r \ge c'' \log\log(1/\epsilon)$ & $r = C'' \log\log(1/\epsilon)$ \\
\hline
\end{tabular}
\end{center}
Note that the necessary and sufficient conditions match up to constant factors for preferential attachment and $d$-regular diffusion trees, implying the existence of a threshold at $k = \Theta\left(\log(1/\epsilon)\right)$ and $r = \Theta\left(\log\log(1/\epsilon)\right)$, respectively. In the case of uniform attachment, our bounds differ by a factor of $\log\log(1/\epsilon)$.

The remainder of the paper is organized as follows: We begin in Section~\ref{SecPrelim} by defining tree centrality and establishing basic properties of centroids. We also define the random growth models that we will discuss in the paper. In Section~\ref{SecPersist1}, we establish persistence of a unique centroid for each of the random growing tree models, with probability 1, and then extend the result to the set of top $K$ central nodes in Section~\ref{SecPersistK}. In Section~\ref{SecHub}, we explore the problem of ensuring persistent centrality of the root node by initializing the random growth model by an appropriate seed tree. We establish upper and lower bounds on the size of the initial seed as a function of the error probability of persistence. We conclude in Section~\ref{SecDiscussion} by discussing several interesting open problems.


\section{Preliminaries}
\label{SecPrelim}

We begin by introducing the notion of centrality in trees, as well as the probabilistic models of random growing trees that we will study in this paper.

\subsection{Centrality}

Let the set of vertices of a tree $T$ be denoted by $V(T)$. A rooted tree is denoted by $(T, u)$, with $u \in V(T)$. The subtree starting from $v$ in a rooted tree $T$ is denoted by $T_{v \downarrow}$. Define the function $\psi_T: V(T) \to \mathbb N$ by
\begin{equation}
\psi_T(u) = \max_{v \in V(T) \setminus \{u\}} |(T, u)_{v\downarrow}|.
\end{equation}
Thus, $\psi_T(u)$ is size of the largest subtree of the rooted tree $(T, u)$. 
\begin{definition}
Given a tree $T$, a vertex $u \in V(T)$ is called a \emph{centroid} if 
$$\psi_T(u) \leq \psi_T(v), \quad \text{ for all } v \in V(T).$$
For any two nodes $u$ and $v$, if  $\psi_T(u) \leq \psi_T(v)$, we say that $u$ is \emph{at least as central as} $v$.
\end{definition}

The first lemma provides a characterization of tree centroids. Similar results have been discovered and rediscovered in a number of papers \cite{Knuth73, Kang75, ShahEtAl11, Zel68}. We include a proof here for completeness.

\begin{lemma}\label{lemma: shah}
For a tree $T$ on $n$ vertices, the following statements hold:
\begin{itemize}
\item[(i)] If $v^*$ is a centroid, then
$$\psi_T(v^*) \leq \frac{n}{2}.$$
\item[(ii)] $T$ can have at most two centroids.
\item[(iii)] If $u^*$ and $v^*$ are two centroids, then $u^*$ and $v^*$ are adjacent vertices. Furthermore,
\begin{equation*}
\psi_T(u^*) = |(T, u^*)_{v^*\downarrow}|, \quad \text{and} \quad \psi_T(v^*) = |(T, v^*)_{u^*\downarrow}|.
\end{equation*}
\end{itemize}
\end{lemma}

\begin{proof}
It is easy to check that the results hold for $n=2$, so we assume that $n \geq 3$ for the rest of the proof. Let $v^*$ be a centroid of $T$. Let the neighbors of $v^*$ be the vertices $\{a_1, \dots, a_k\}$. Note that if $k=1$, then $\psi_T (v^*) = n-1$, and one can check that $\psi_T(a_1) < n-1$. This contradicts the assumption that $v^*$ is a centroid. Hence, we must have $k \geq 2$. Denote
\begin{align*}
|(T, v^*)_{a_i \downarrow}| = r_i, \quad \text{for} \quad 1\leq i\leq k.
\end{align*}
Without loss of generality, assume $r_1 \geq r_2 \geq \dots \geq r_k$. Thus, we have $\psi_T(v^*) = r_1$. The key step is to look at the subtrees of $(T, a_1)$. If $b \neq v^*$ is any neighbor of $a_1$, we have $(T, a_1)_{b \downarrow} \subset (T, v^*)_{a_1\downarrow}$. Thus, $|(T, a_1)_{b \downarrow} | < r_1$. Therefore, to ensure that $\psi_T(v^*) \leq \psi_T(a_1)$, we must have
\begin{equation*}
|(T, a_1)_{v^*\downarrow}| \geq r_1,
\end{equation*}
which simplifies to
\begin{equation}\label{eq: r}
1 + \sum_{i=2}^k r_i \geq r_1.
\end{equation}
Adding $r_1$ to both sides, and noting that $\sum_{i=1}^k r_i +1 = n$, we conclude that $r_1 \leq n/2$, which is part (i).

To show part (ii), note that none the vertices in the set $\cup_{i=2}^k (T, v^*)_{a_i \downarrow}$ can be centroids, since for any $u \in \cup_{i=2}^k (T, v^*)_{a_i \downarrow}$, we have
\begin{equation*}
\psi_T(u) > |(T, u)_{a_1 \downarrow}| = |(T, v^*)_{a_1 \downarrow}| = \psi_T(v^*).
\end{equation*}
Thus, any centroids apart from $v^*$ must lie in $(T, v^*)_{a_1 \downarrow}$. For any node $u \in (T, v^*)_{a_1 \downarrow}$ such that $u \neq a_1$, we have
\begin{equation*}
\psi_T(u) > |(T, u)_{v^*\downarrow}| = 1 + \sum_{i=2}^k r_i \geq r_1,
\end{equation*}
where the second inequality follows from inequality \eqref{eq: r}. Thus, the only potential centroid apart from $v^*$ is the node $a_1$, which proves part (ii). Note that $a_1$ can be a centroid if and only if
\begin{equation*}
\psi_T(a_1) = |(T, a_1)_{v^*\downarrow}| = 1+ \sum_{i=2}^k r_i = r_1 = |(T, v^*)_{a_1\downarrow}| = \psi_T(v^*).
\end{equation*}
This proves part (iii) and concludes the proof.
\end{proof}

We now turn our attention to growing trees. We have the following definition:

\begin{definition}
A collection of trees $\{T_n\}_{n \ge 1}$ is called a \emph{sequence of growing trees} if $T_n$ has $n$ nodes, and $T_{n+1}$ is obtained from $T_n$ by adding a single vertex that is attached to a vertex of $T_n$ by a single new edge.
\end{definition}

The next lemma concerns the evolution of centroids in sequences of growing trees.

\begin{lemma}\label{lemma: main} 
Consider a sequence of growing trees $\{T_n\}_{n \ge 1}$, with vertices labeled in order of appearance, so $V(T_n) = \{v_1, v_2, \dots, v_n\}$. Let $v^*(n)$ be a centroid of $T_n$, and let $n > 2$. If at some time $N > n$, the node $v_{n+1}$ becomes at least as central as $v^*(n)$; i.e., if 
$$\psi_{T_N}(v_{n+1}) \leq \psi_{T_N}(v^*(n)),$$
then for some $n +1 \leq M \leq N$, we must have
\begin{equation}
\label{EqnHibernate}
\psi_{T_M}(v^*(n)) = \psi_{T_M}(v_{n+1}),
\end{equation}
and
\begin{equation}
\label{EqnZop}
|(T_M, v^*(n))_{v_{n+1}\downarrow}| = |(T_M, v_{n+1})_{v^*(n)\downarrow}|.
\end{equation}
\end{lemma}

\begin{proof} 
Note that for any fixed vertex $v$, the size of the largest subtree of $(T_n, v)$ either increases by 1 or remains constant when the new vertex $v_{n+1}$ joins $T_n$. Thus, $\psi_{T_n}(v)$ increases by at most 1 at for each time step. At time $n+1$, we have $\psi_{T_{n+1}}(v_{n+1}) = n > \psi_{T_{n+1}}(v^*(n))$, where the inequality follows from Lemma~\ref{lemma: shah}. Note that the difference $\psi_{T}(v_{n+1}) - \psi_{T}(v^*(n))$ changes by at most 1 as the tree $T$ evolves at each time step. Hence, if the difference becomes nonpositive at some time $n = N$, there must exist a time $M \leq N$ when the difference is exactly zero. This implies that there exists an $M$ such that $k+1\leq M \leq N$, so equation~\eqref{EqnHibernate} holds.

Now consider the subtrees of $(T_M, v^*(n))$ and $(T_M, v_{n+1})$. Let $(v^*(n), u_1, u_2, \dots, u_\ell, v_{n+1})$ denote the path from $v^*(n)$ to $v_{n+1}$, where $\ell \ge 0$. Suppose the largest subtree of $(T_M, v_{n+1})$ is $(T_M, v_{n+1})_{w\downarrow}$, for some $w \neq u_\ell$.  It is easy to see that
\begin{align*}
\psi_{T_M}(v^*(n)) &\geq |(T_M, v^*(n))_{u_1 \downarrow} |\\
&\geq \ell + 1 + |(T_M, v^*(n))_{w\downarrow}| \\
&=  \ell + 1 + |(T_M, v_{n+1})_{w\downarrow}| \\
&= \ell + 1 + \psi_{T_M}(v_{n+1}),
\end{align*}
which contradicts equation~\eqref{EqnHibernate}. Thus, the largest subtree of $(T_M, v_{n+1})$ must be $(T_M, v_{n+1})_{u_\ell \downarrow}$. Using the same argument for $v^*$, we conclude that the largest subtree of $(T_M, v^*(n))$ must be $(T_M, v^*(n))_{u_1 \downarrow}$. By equation~\eqref{EqnHibernate}, we then have $|(T_M, v_{n+1})_{u_\ell \downarrow}| = |(T_M, v^*(n))_{u_1 \downarrow}|$. It is then easy to see that equation~\eqref{EqnZop} holds, as well.
\end{proof}

We also have the following useful result:

\begin{lemma}\label{lemma: n/2}
Let $\{T_n\}_{n \ge 1}$ be a sequence of growing trees, with $V(T_n) = \{v_1, \dots, v_n\}$. At time $n+1$, we have the inequality
\begin{align*}
|(T_{n+1}, v_{n+1})_{v^*(n)\downarrow}| \geq \frac{n}{2}.
\end{align*}
\end{lemma}

\begin{proof}
As before, let $(v^*(n), u_1, u_2, \dots, u_\ell, v_{n+1})$ denote the path from $v^*(n)$ to $v_{n+1}$. We have the equality
\begin{equation*}
|(T_{n+1}, v_{n+1})_{v^*(n)\downarrow}| = (n+1) - |(T_{n+1}, v^*(n))_{u_1 \downarrow}|.
\end{equation*}
From Lemma \ref{lemma: shah}, we have
\begin{equation*}
|(T_{n+1}, v^*(n))_{u_1 \downarrow}| \leq 1 + \psi_{T_{n}}(v^*(n)) \leq 1 + \frac{n}{2}.
\end{equation*}
Substituting, we arrive at
\begin{align*}
|(T_{n+1}, v_{n+1})_{v^*(n)\downarrow}| = (n+1) - |(T_{n+1}, v^*(n))_{u_1 \downarrow}| \geq (n+1) - \frac{n}{2} - 1 = \frac{n}{2}.
\end{align*}
\end{proof}


\subsection{Random growing trees}

We now describe the probabilistic models generating the sequences of growing trees to be considered in this paper. Accordingly, we have the following definitions:

\begin{definition} [Uniform attachment]
A sequence of growing trees $\{T_n\}_{n \geq 1}$ is generated by a \emph{uniform attachment process} if
\begin{itemize}
\item[(a)] $T_1$ consists of a single vertex $v_1$, and
\item[(b)] $T_{n+1}$ is created from $T_n$ by introducing a new vertex $v_{n+1}$ and attaching it to a vertex in $T_n$ uniformly at random; i.e., with probability $1/n$ to each existing node.
\end{itemize}
\end{definition}

\begin{definition} [Preferential attachment]
A sequence of growing trees $\{T_n\}_{n \geq 1}$ is generated by a \emph{preferential attachment process} if
\begin{itemize}
\item[(a)] $T_1$ consists of a single vertex $v_1$, and
\item[(b)] $T_{n+1}$ is created from $T_n$ by introducing a new vertex $v_{n+1}$ and attaching it to a random vertex in $T_n$, with probability $\frac{\text{deg}(v_i)}{\sum_{j=1}^n \text{deg}(v_j)}$ for vertex $v_i \in V(T_n)$.
\end{itemize}
\end{definition}

\begin{definition} [$d$-regular tree diffusion]
For $d \geq 2$, let $\cG$ be an infinite $d$-regular tree; i.e., a tree where each vertex has degree $d$. A sequence of growing trees $\{T_n\}_{n \ge 1}$ is generated by a \emph{$d$-regular diffusion process} if
\begin{itemize}
\item[(a)] $T_1$ consists of a single vertex $v_1 \in \cG$, and
\item [(b)] if $\cN(T_n)$ denotes the set of neighbors of vertices in $T_n$ not contained in $V(T_n)$, the tree $T_{n+1}$ is created from $T_n$ by picking $v_{n+1} \in \cN(T_n)$ uniformly at random, and adding it to $T_n$ together with its connecting edge.
\end{itemize}
\end{definition}

The models described above are well-studied~\cite{Bai75, BarAlb99, Jack08} and are also examples of plane-oriented recursive trees~\cite{Drm09}.


\section{Existence of a persistent centroid}
\label{SecPersist1}

In this section, we show that with probability 1, a single centroid emerges for each sequence of random growing trees described in the previous section. We have the following definition:

\begin{definition}
A vertex $v^* \in \cup_{n=1}^\infty V(T_n)$ is a \emph{persistent centroid} for the sequence of growing trees $\{T_n\}_{n \ge 1}$ if there exists $N$ such that for all $n \ge N$, the vertex $v^*$ is a centroid of $T_n$.
\end{definition}

For a tree $T_n$ on $n$ vertices, let $\cC(T_n)$ denote the set of centroids of $T_n$. Note that by Lemma \ref{lemma: shah}, we have $|\cC(T_n)| \in \{1, 2\}$. Define 
\begin{equation*}
\cC_\text{tot} = \cup_{n = 1}^\infty \cC(T_n),
\end{equation*}
so $\cC_\text{tot}$ is the set of all vertices that are centroids at any point in time.

\begin{remark}
\label{RemD2}
Throughout this section, we will assume that $d \ge 3$ in the case of $d$-regular trees. Indeed, for $d=2$, the set $\cC_{\text{tot}}$ is infinite with probability 1. This is because diffusion on a 2-regular tree produces a sequence of line graphs, so the midpoint is the unique centroid if the number of vertices is odd, and the middle two nodes constitute the centroid set if the number of vertices is even. Since the number of vertices alternates between odd and even, a unique centroid cannot exist. Moreover, it is impossible for any node $v$ to be a centroid for all but finitely many time steps: If $v$ becomes the centroid at some time $N$, it will be a centroid at time $n \ge N$ if and only if the number of additional nodes added to the left of $v$ differs from the number of additional nodes added to the right of $v$ by at most 1. Since nodes are added to the left or to the right with equal probability, it follows from properties of a simple random walk that with probability 1, centrality cannot persist.
\end{remark}

We first show that the total number of vertices that have ever been centroids is finite with probability 1.

\begin{lemma}\label{lemma: finite}
For the preferential and uniform attachment models, and for the $d$-regular diffusion tree with $d \ge 3$, we have $|\cC_\text{tot}| < \infty$, with probability 1.
\end{lemma}
\begin{proof}
We aim to show that any node joining the tree ``sufficiently late" has a very small chance of becoming a centroid at some future time. We first explain how Lemma \ref{lemma: main} may be leveraged to substantially simplify the proof of this fact. 

Let $v^*(k)$ be a centroid of $T_k$. Suppose the node $v_{k+1}$, which joins $T_k$ at time $k+1$, becomes a centroid of $T_{N}$ for some large enough $N$. Then $v_{k+1}$ must be at least as central as $v^*(k)$ at time $N$. Consider the evolution of the vector $(\psi_{T_n}(v_{k+1}), \psi_{T_n}(v^*(k)))$ with $n$. At time $n = k+1$, this vector is equal to $(k, \psi_{T_{k+1}}(v^*(k)))$, which is a point below the diagonal in $\mathbb N \times \mathbb N$. At each time step, this vector may perform one of four moves: move one step to the right, move one step above, move one step diagonally, or remain stationary. At time $N$, this walk is either on or above the diagonal, since $v_{k+1}$ is at least as central as $v^*(k)$. To bound the probability of that event, we must keep track of the largest subtrees of $(T_n, v^*(k))$ and $(T_n, v_{k+1})$, as well as the location of the new node $v_{n+1}$. However, Lemma \ref{lemma: main} makes it possible to ignore complicated tree dynamics: First, the lemma indicates that we may bound the probability of the random walk crossing the diagonal by the probability of it reaching the diagonal at some time $M$. Second, the random walk reaches the diagonal at time $M$ if and only if $|(T_M, v^*(k))_{v_{k+1}\downarrow}| = |(T_M, v_{k+1})_{v^*(k) \downarrow}|$. Thus, we may simply keep track of the random vector $\Big(|(T_n, v^*(k))_{v_{k+1}\downarrow}|, |(T_n, v_{k+1})_{v^*(k) \downarrow}| \Big)$, for $n\geq k+1$, and bound the probability of it reaching the diagonal. The evolution of the latter vector is significantly easier to track, since the dynamics of the tree are largely ignored. This random walk may either move one step to the right or one step up (it can also stay in the same place, but we may simply ignore those time steps).

For a point $(i,j)$, let the probability of moving up be $U(i,j)$ and of moving right be $R(i,j)$. For the growing random trees we consider, these probabilities are given by:
\begin{enumerate}
\item \textbf{Preferential attachment:} The probability of a new node joining either  $(T_n, v^*(k))_{v_{k+1}\downarrow}$ or $ (T_n, v_{k+1})_{v^*(k) \downarrow}$ is proportional to the total number of edges incident upon the vertices in the corresponding subtrees. Thus, the probabilities governing the random walk are given by
\begin{equation*}
R(i, j) = \frac{2i-1}{2(i+j-1)}, \quad \text{ and } \quad U(i,j) = \frac{2j-1}{2(i+j-1)}.
\end{equation*}

\item \textbf{Uniform attachment:} Here, the probability of a new node joining either  $(T_n, v^*(k))_{v_{k+1}\downarrow}$ or $ (T_n, v_{k+1})_{v^*(k) \downarrow}$ is proportional to the sizes of these subtrees. Thus, the probabilities are given by
\begin{equation*}
R(i, j) = \frac{i}{i+j}, \quad \text{ and } \quad U(i,j) = \frac{j}{i+j}.
\end{equation*}

\item \textbf{Diffusion on a $d$-regular tree:} In this model, the probability of a new node joining either  $(T_n, v^*(k))_{v_{k+1}\downarrow}$ or $ (T_n, v_{k+1})_{v^*(k) \downarrow}$ is proportional to the respective neighborhood sizes $\cN\Big((T_n, v^*(k))_{v_{k+1}\downarrow} \Big)$ and $\cN\Big(  (T_n, v_{k+1})_{v^*(k) \downarrow} \Big)$. These numbers depend only on the size of the corresponding subtrees, and we can write the probabilities as
\begin{equation*}
R(i, j) = \frac{(d-2)i + 1}{(d-2)(i+j)+2}, \quad \text{ and } \quad U(i,j) = \frac{(d-2)j+1}{(d-2)(i+j)+1}.
\end{equation*}
\end{enumerate}

Note that in all the examples above, the probability of joining a subtree is proportional to an affine function of the size the subtree. These are precisely the types of random walks discussed in Lemma \ref{lemma: galashin} in Appendix~\ref{AppRW}. Consider the events
$$H_k = \{\text{$v_{k+1}$ becomes at least as central as $v^*(k)$ at some future time} \}.$$ It is enough to show that only finitely many events $H_k$ occur, since this ensures that new vertices are added to $\cC_{tot}$ only finitely many times.

As described in Lemma \ref{lemma: main}, the probability of event $H_k$ is the probability that the random walk $\Big(|(T_{n}, v_{k+1})_{v^*(k)\downarrow}|, |(T_{n}, v^*(k))_{v_{k+1}\downarrow}|\Big)$ reaches the diagonal at some point. Note that at $n = k+1$, Lemma \ref{lemma: n/2} gives
\begin{align*}
|(T_{k+1}, v_{k+1})_{v^*(k)\downarrow}| \geq k/2,
\end{align*}
whereas $|(T_{k+1}, v^*(k))_{v_{k+1}\downarrow}| = 1$. By Lemma \ref{lemma: galashin} in Appendix~\ref{AppRW} we then have
\begin{equation*}
\mathbb P(H_k) \leq \max_{A \geq k/2} \frac{P(A)}{2^A} \stackrel{(a)} =  \frac{P(k/2)}{2^{k/2}},
\end{equation*}
where $P(A)$ is a fixed polynomial, and equality $(a)$ holds for all large enough $k \ge K_0$. We then have
\begin{align*}
\sum_{k=1}^\infty \mathbb P(H_k) &= \sum_{k=1}^{K_0-1} \mathbb P(H_k)  + \sum_{k = K_0}^\infty \mathbb P(H_k) \\
&\leq  \sum_{k=1}^{K_0-1} \mathbb P(H_k)  + \sum_{k = K_0}^\infty \frac{P(k/2)}{2^{k/2}}\\
&< \infty. 
\end{align*}
Using the Borel-Cantelli lemma, we conclude that with probability 1, only finitely many events $H_k$ occur, completing the proof.
\end{proof}

To establish the existence of a persistent centroid, we still need to show that the elements in $\cC_\text{tot}$ do not keep replacing each another as centroids. Our next lemma establishes this fact. The result of the lemma may clearly be extended to any finite collection of vertices, showing that the centrality of all the vertices in the set will eventually separate. For any two vertices $u$ and $v$, we define
\begin{equation*}
\cD_\psi(u,v) := \{n \mid \psi_{T_n}(v) = \psi_{T_n}(u)\}.
\end{equation*}

\begin{lemma}\label{lemma: piggy}
For each of the models described in Lemma~\ref{lemma: finite}, and for any two distinct vertices $u$ and $v$, we have $|\cD_\psi(u,v)| < \infty$, with probability 1.
\end{lemma}

\begin{proof}
By Lemma \ref{lemma: main}, it suffices to show that with probability 1, the random walk defined by $(X_n, Y_n) := \Big( |(T_n, v)_{u \downarrow}|, |(T_n, u)_{v \downarrow}| \Big)$ touches the diagonal only finitely many times. Without loss of generality, we assume that vertex $v$ is born after vertex $u$. Thus, the random walk starts when vertex $v$ is born, and the starting point is $ \Big( |(T_n, v)_{u \downarrow}|, 1 \Big) := (A, 1)$. As in Lemma \ref{lemma: galashin} in Appendix~\ref{AppRW}, let the probability that a vertex is added to a subtree of size $i$ be proportional to $i + \beta/\alpha$, where $1 + \beta/\alpha>0$. The evolution of the vector $(X_n, Y_n)$  then follows a standard P\'olya urn model, and by almost sure convergence of martingale sequences, combined with standard distributional convergence results~\cite{Pem07}, we have
\begin{align*}
\frac{X_n}{X_n + Y_n} \stackrel{\text{a.s.}} \to \xi \sim \text{Beta}(A+ \beta/\alpha, 1 + \beta/\alpha).
\end{align*}
By absolute continuity of the Beta distribution, we have $\mathbb P (\xi = 1/2) = 0$. Since the fraction converges to $\xi \neq 1/2$, with probability 1 it can equal $1/2$ only finitely many times. This proves the lemma.
\end{proof}

Lemmas \ref{lemma: finite} and \ref{lemma: piggy} together imply the existence of a single persistent centroid. This is summarized in the following theorem:

\begin{theorem*}
\label{ThmPersist}
For the preferential and uniform attachment models, and for the $d$-regular diffusion tree with $d \ge 3$, there exists a time $N$ and a node $v^* \in T_N$ such that $v^*$ is the unique centroid of $T_n$ for all $n \ge N$, with probability 1.
\end{theorem*}

\begin{proof}
By Lemma \ref{lemma: finite}, the set of vertices that are ever centroids is finite. Clearly, if a single centroid does not persist, there exist at least two vertices that surpass each other infinitely often in terms of centrality. However, Lemma \ref{lemma: piggy} rules out such a scenario, implying the persistence of a single centroid.
\end{proof}


\section{Persistence of the top $K$ central nodes}
\label{SecPersistK}

We now extend the result of the previous section to establish persistence of the top $K$ central nodes. The main theorem of this section has an important consequence concerning root-finding algorithms that generate a confidence set for the initial vertex of the random growing tree~\cite{BubEtal14, KhiLoh15}. As discussed in more detail following the statement of Theorem~\ref{ThmPersistK}, the theorem implies the eventual stability of the confidence set selected according to the function $\psi$.

For $n \geq K$, let $\cK_n = \{\nu_1(n), \dots, \nu_K(n)\}$ denote the set of vertices of $T_n$ that are most central in the following sense: For every vertex $v \notin \cK_n$, we have the inequality
\begin{equation*}
\psi_{T_n}(v) \geq \max_{\nu_i \in \cK_n} \psi_{T_n}(\nu_i(n)).
\end{equation*}
The set $\cK_n$ contains the $K$ vertices of $T_n$ having the smallest largest subtrees, with ties being broken arbitrarily. We assume without loss of generality that 
$$\psi_{T_n}(\nu_1(n)) \leq \psi_{T_n}(\nu_i(n)) \leq \dots \leq \psi_{T_n}(\nu_K(n)).$$
The main result of this section is to show that with probability 1, the set $\cK_n$ also has the persistence property. In other words, there exist vertices $\{v_1^*, \dots, v_K^*\}$ and some $N$ such that for all $n \ge N$, the $v_i^*$'s are the unique top $K$ central nodes in $T_n$.

Our first lemma establishes that even the least central vertex in $\cK_n$ has its largest subtree size ``not too large"---i.e., of size bounded by a linear function not identically equal to $n$. The proof requires a P\'{o}lya urn analysis that tracks the number of vertices in the subtrees connected to the first $K$ nodes in each of the random growth models. We will again restrict our attention in the $d$-regular diffusion case to $d \ge 3$, since as discussed in Remark~\ref{RemD2}, persistence cannot occur in the case $d = 2$.

\begin{lemma}\label{lemma: cafe 12}
For the preferential and uniform attachment models, and for the $d$-regular diffusion tree with $d \ge 3$, there exists a continuous random variable $\xi$ satisfying $\mathbb P(\xi < 1) = 1$ and $$\psi_{T_n}(\nu_K(n)) \leq \xi n,$$
almost surely, for all $n \geq K$.
\end{lemma}
\begin{proof}
Let $\{v_1,\dots, v_K\}$ denote the first $K$ vertices, i.e., the vertices of $T_K$. For any $n \geq K$, we have
\begin{equation*}
\max_{1 \leq i \leq K} \psi_{T_n}(v_i) \geq \max_{1 \leq i \leq K} \psi_{T_n}(\nu_i(n)) = \psi_{T_n}(\nu_K(n)).
\end{equation*}
Thus, it suffices to derive an upper bound for $\max_{1 \le i \le K} \psi_{T_n}(v_i)$.

For $1 \leq i \leq n$, let $T_{i, n}$ be the tree in the forest formed by removing all the edges between $\{v_1, \dots, v_K\}$ in $T_n$. Clearly,
\begin{equation*}
\psi_{T_n}(v_i) \leq \max(|T_{i,n}|, n-|T_{i,n}|), \quad \text{for} \quad 1\leq i\leq K.
\end{equation*}
Thus,
\begin{align}
\label{EqnShrew}
\max_{1 \leq i \leq K} \psi_{T_n}(v_i) &\leq \max(|T_{1,n}|, n-|T_{1,n}|, \dots, |T_{K,n}|, n-|T_{K,n}|) \notag \\
&= \max(\max_{1 \leq i \leq K} |T_{i,n}|, n-\min_{1\leq i \leq K} |T_{i,n}|) \notag \\
&\le n-\min_{1\leq i \leq K} |T_{i,n}|.
\end{align}
Thus, an appropriate lower bound on $\min_{1\leq i \leq K} |T_{i,n}|$ will provide the desired upper bound. We establish a random linear lower bound for each of the growing graphs separately, beginning with uniform attachment. Apart from being easier to analyze, it will illustrate the idea that we will use in the other two cases.
\begin{enumerate}
\item \textbf{Uniform attachment:} The vector $(|T_{1,n}|, \dots, |T_{K,n}|)$ evolves according to a standard P\'olya urn process with replacement matrix $I_K$ and starting state $(1, 1, \dots, 1)$. Thus,  
$$\left(\frac{|T_{1,n}|}{n}, \dots, \frac{|T_{K,n}|}{n} \right) \stackrel{a.s.} \to (C_1, \dots, C_K) \sim \text{Dirichlet}(1, 1, \dots, 1).$$ By the continuous mapping theorem, we conclude that
$$\frac{1}{n} \min_{1\leq i \leq K} |T_{i,n}| \stackrel{a.s.} \to C =\min(C_1, \dots, C_K),$$
where $C$ is a continuous random variable taking values in  $[0,1] $. Taking inverses, we then have 
$$\frac{n}{\min_{1\leq i \leq K} |T_{i,n}|}  \stackrel{a.s.} \to \frac{1}{C}.$$ 
Note that $1/C$ does not have a point mass at infinity, since $C$ is a continuous random variable. This almost sure convergence implies the existence of a random variable $\hat \xi$ such that \mbox{$\mathbb P( \hat \xi < \infty) = 1$}, and which bounds $\frac{n}{\min_{1\leq i \leq K} |T_{i,n}|}$ almost surely, for all $n$. Hence, $$\min_{1\leq i \leq K} |T_{i,n}| \geq \frac{n}{\hat \xi},$$
for all $n \geq K$.  Substituting into inequality~\eqref{EqnShrew}, we then obtain
\begin{align*}
\max_{1 \leq i \leq K} \psi_{T_n}(v_i) = n-\min_{1\leq i \leq K} |T_{i,n}| \leq n(1-1/\hat \xi) := n\xi.
\end{align*}
Since $\hat \xi < \infty$ with probability 1, we have $\xi < 1$ with probability 1. This concludes the proof.

\item \textbf{Preferential attachment:} At time $K$, the number of possible structures of $T_K$ is finite. We denote the set of all possible trees at time $K$ by $\text{Trees}_K = \{t_1,t_2, \dots, t_\kappa\}$, where \mbox{$\kappa = |\text{Trees}_K|$}. Let $\mathbb P (T_K= t_i) = p_i$. Also let $S_{\ell,n}$ denote the degree sum of the vertices in $T_{\ell,n}$. Conditioned on $T_K = t_i$, the vector $(S_{1,n}, \dots, S_{K,n})$ evolves according to a P\'{o}lya urn process with replacement matrix $2I_K$ and initial configuration $(\deg(v_1), \dots, \deg(v_K))$, corresponding to the degrees of the vertices $(v_1, \dots, v_K)$ in $t_i$. Hence, conditioned on $T_K = t_i$, we have
\begin{equation*}
\left(\frac{S_{1,n}}{2(n-1)}, \dots, \frac{S_{K,n}}{2(n-1)}\right) \stackrel{a.s.}{\to} \text{Dirichlet}\left(\frac{\deg(v_1)}{2}, \dots, \frac{\deg(v_K)}{2}\right).
\end{equation*}
Furthermore,
\begin{equation*}
S_{\ell,n} = 2(|T_{\ell,n}| - 1) + \deg(v_\ell), \qquad \forall 1 \le \ell \le K,
\end{equation*}
so it is easy to see that
\begin{equation}
\label{EqnBunny}
\left(\frac{|T_{1,n}|}{n}, \dots, \frac{|T_{K,n}|}{n} \right) \stackrel{a.s.}  \to (C^i_1, \dots, C^i_K) \sim  \text{Dirichlet}\left(\frac{\text{deg}(v_1)}{2}, \dots, \frac{\text{deg}(v_K)}{2}\right),
\end{equation}
as well. By the continuous mapping theorem, equation~\eqref{EqnBunny} implies the almost sure convergence
$$\frac{1}{n} \min_{1\leq i \leq K} |T_{i,n}| \stackrel{a.s.} \to \min(C^i_1, \dots, C^i_K),$$
so $$\frac{n}{\min_{1\leq i \leq K} |T_{i,n}|}  \stackrel{a.s.} \to \frac{1}{\min(C^i_1, \dots, C^i_K)}.$$
Thus, there exists a continuous random variable $\hat \xi^i$ that bounds $\frac{n}{\min_{1\leq i \leq K} |T_{i,n}|}$, almost surely, for all $n$, so
$$\min_{1\leq i \leq K} |T_{i,n}| \geq \frac{n}{\hat \xi^i},$$
for all $n \geq K$.  Substituting into inequality~\eqref{EqnShrew}, we then obtain
\begin{align*}
\max_{1 \leq i \leq K} \psi_{T_n}(v_i) = n-\min_{1\leq i \leq K} |T_{i,n}| \leq n(1-1/\hat \xi^i).
\end{align*}
Define the random variable $\xi$ to equal $(1-1/\hat \xi^i)$ on the event $\{T_K = t_i\}$. Using a similar argument as in the case of uniform attachment, we have $\hat \xi^i < \infty$ for each $i$ with probability 1, so $\xi < 1$ with probability 1. 

\item \textbf{$d$-regular diffusion:} As in the case of the preferential attachment model, we define the set of all possible trees at time $K$ by $\text{Trees}_K = \{t_1,t_2, \dots, t_\kappa\}$, where $\kappa = |\text{Trees}_K|$ and $\mathbb P (T_K= t_i) = p_i$. Let $U_{\ell,n}$ denote the number of uninfected neighbors of vertices in $T_{\ell,n}$. Conditioned on $T_K = t_i$, the vector $(U_{1,n}, \dots, U_{K,n})$ evolves according to a P\'{o}lya urn process with replacement matrix $(d-2)I_K$ and initial configuration $(d-\deg(v_1), \dots, d-\deg(v_K))$, where $(\deg(v_1), \dots, \deg(v_K))$ again denotes the degrees of $(v_1, \dots, v_K)$ in $t_i$. Then
\begin{equation*}
\left(\frac{U_{1,n}}{(d-2)n}, \dots, \frac{U_{K,n}}{(d-2)n}\right) \stackrel{a.s.}{\to} \text{Dirichlet}\left(\frac{d-\deg(v_1)}{d-2}, \dots, \frac{d-\deg(v_K)}{d-2}\right),
\end{equation*}
implying that
\begin{equation*}
\left(\frac{|T_{1,n}|}{n}, \dots, \frac{|T_{K,n}|}{n} \right) \stackrel{a.s.}  \to (C^i_1, \dots, C^i_K) \sim  \text{Dirichlet}\left(\frac{d-\text{deg}(v_1)}{d-2}, \dots, \frac{d-\text{deg}(v_K)}{d-2}\right).
\end{equation*}
The remainder of the analysis proceeds exactly as in the case of the preferential attachment model.
\end{enumerate} 
This completes the proof of the lemma.
\end{proof}

We now define the collection of vertices that ever enter the set of top $K$ most central nodes. Let $\cK_n'$ denote the set $\cK_n$ augmented with any additional vertices that are at least as central as $\nu_K(n)$ in $T_n$, and let
\begin{equation*}
\cK_\text{tot} := \cup_{n = 1}^\infty \cK_n'.
\end{equation*}
We have the following lemma, the analog of Lemma~\ref{lemma: finite}:

\begin{lemma}
\label{LemFiniteK}
In the same setting as Lemma~\ref{lemma: cafe 12}, we have $|\cK_{tot}| < \infty$, with probability 1.
\end{lemma}

\begin{proof}
Consider the set of events
$$B_M = \bigcap_{n \ge K} \{\psi_{T_n}(\nu_K(n)) \leq nM\},$$
for any real $M \in (0,1)$. Thus, $B_M$ is the event that the least central node in $\cK_n$, i.e., $\nu_K(n)$, has its largest subtree upper-bounded by $nM$ at every time $n$. Now consider the event
$$H_n = \{\exists \ell: v_{n+1} \in \cK_\ell'\}.$$
Thus, $H_n$ is the event that $v_{n+1}$ becomes at least as central as one of the top $K$ central nodes at some point in the future. On the event $H_n$, we must have
$$\psi_{T_\ell}(v_{n+1}) \leq \max_{1\leq i \leq K} \psi_{T_\ell}(\nu_i(n)).$$
Now define the event
$$E_i = \{v_{n+1} \text{ becomes at least as central as } \nu_i(n) \text{ at some future point}\}.$$
We have the bound
\begin{align*}
\mathbb P(B_M \cap H_n) &\leq \mathbb P\left(B_M \cap \left(\bigcup_{i=1}^K E_i \right)\right) \leq \sum_{i=1}^K \mathbb P(B_M \cap E_i).
\end{align*}
By Lemma \ref{lemma: main}, we may control the probability $\mathbb P(B_M \cap E_i)$ by bounding the probability that the random walk $\Big(|(T_\ell, v_{n+1})_{\nu_i(n)\downarrow}|, |(T_\ell, \nu_i(n))_{v_{n+1}\downarrow}|\Big)$ reaches the diagonal. Note that this walk starts from the point $\Big(|(T_{n+1}, v_{n+1})_{\nu_i(n)\downarrow}|, 1\Big)$ at time $\ell = n+1$. If $(\nu_i(n), u_1, \dots, v_{n+1})$ is the path from $\nu_i(n)$ to $v_{n+1}$, then on the event $B_M$, we have
\begin{align*}
|(T_{n+1}, v_{n+1})_{\nu_i(n)\downarrow}| &= n - |(T_n, \nu_i(n))_{u_1\downarrow}|\\
&\geq n - \psi_{T_n}(\nu_i(n))\\
&\geq n - \psi_{T_n}(\nu_K(n)\\
&\geq n-Mn \\
&= n(1-M).
\end{align*} 
Thus, the starting point lies below the diagonal and to the right of the point $\big((1-M)n, 1\big)$. Lemma~\ref{lemma: galashin} in Appendix~\ref{AppRW} then implies that
\begin{align*}
\mathbb P(B_M \cap H_n) \leq K \cdot \max_{A \geq (1-M)n} \frac{P(A)}{2^A} \stackrel{(a)} = K \cdot \frac{P\big((1-M)n\big)}{2^{(1-M)n}},
\end{align*}
where $(a)$ holds for all large enough $n$. The expression on the right-hand side form a convergent series in $n$. Applying Borel-Cantelli lemma, we conclude that for all $M$, the event $H_n \cap B_M$ occurs finitely often, with probability 1. Furthermore, Lemma \ref{lemma: cafe 12} implies that $\mathbb P(B_M)\to 1$ as $M \to 1$, since the random variable $\xi$ appearing in the lemma does not have a point mass at 1. Therefore, with probability 1, the events $H_n$ can can occur only finitely often, which implies the desired statement.
\end{proof}

The stability of the set of top $K$ central nodes then follows by combining Lemmas~\ref{LemFiniteK} and \ref{lemma: piggy}, as in the proof of Theorem~\ref{ThmPersist}:

\begin{theorem*}
\label{ThmPersistK}
For the preferential and uniform attachment models, and for the $d$-regular diffusion tree with $d \ge 3$, with probability 1, there exists a time $N$ and a collection $\{\nu_1^*, \dots, \nu_K^*\} \subseteq T_N$ such that $\{\nu_1^*, \dots, \nu_K^*\}$ are the $K$ most central nodes of $T_n$, for all $n \ge N$.
\end{theorem*}

As mentioned at the beginning of the section, Theorem~\ref{ThmPersistK} has important implications for root-finding algorithms in random growing trees: One may obtain confidence sets for the root node in uniform and preferential attachment models~\cite{BubEtal14} and $d$-regular diffusion trees~\cite{KhiLoh15} by selecting the nodes that minimize the maximum subtree estimator $\psi$. Furthermore, the size of the confidence set may be taken as a fixed function $K(\epsilon)$ of the error probability $\epsilon$, and does not need to grow with $n$. Theorem~\ref{ThmPersistK} implies that the confidence sets constructed in this manner will almost surely stabilize after a finite time, showing that the confidence set construction is in some sense robust.

 
\section{Ensuring centrality of the root node}
\label{SecHub}

The results from the earlier sections indicate that any fixed node has some finite probability of eventually becoming the persistent centroid of a random growing tree. We consider the special case of the root node, i.e., the first vertex $v_1$, and ask the question: Can we ensure that $v_1$ is the persistent centroid of the random growing tree? Note that the probability of the complementary event is at least $1/2$, since there is no way to distinguish nodes $v_1$ and $v_2$. However, in the preferential and uniform attachment graphs, we may boost the probability of $v_1$ being the persistent centroid by initializing the tree with a ``hub" centered at $v_1$ of size $k$. In other words, the graph begins with a star configuration in which the nodes $\{v_2, \dots, v_{k+1}\}$ are all attached to $v_1$. In the case of a $d$-regular tree diffusion, the bounded degree makes it impossible to create a large hub at $v_1$. Hence, we instead begin with the subtree consisting of all nodes at a distance at most $r$ from $v_1$. As a function of $\epsilon$, we derive bounds on the necessary and sufficient size hub size $k$ (for preferential and uniform attachment) and the radius $r$ (for a $d$-regular diffusion) to ensure the persistent centrality of $v_1$ with probability $1-\epsilon$.

We begin by deriving necessary conditions.

\begin{theorem*}\label{ThmHubNec}
The following conditions are necessary to ensure that $v_1$ is the persistent central node, where $C, C'$, and $C''$ are appropriate constants.
\begin{itemize}
\item[(i)] Preferential attachment: The hub size $k$ is at least $C\log(1/ \epsilon)$.
\item[(ii)] Uniform attachment: The hub size $k$ is at least $C'\frac{\log(1/\epsilon)}{\log \log (1/\epsilon)}$.
\item[(iii)] $d$-regular tree diffusion: Suppose $d \ge 3$. The radius $r$ is at least $C'' \log \log (1/\epsilon)$.
\end{itemize}
\end{theorem*}

\begin{proof}
We begin by analyzing the preferential and uniform attachment models. Let $P_k$ denote the probability that the next $k-1$ vertices $\{v_{k+2}, \dots, v_{2k}\}$ all join vertex $v_2$. Since the graph will then be symmetric with respect to $v_1$ and $v_2$, the probability of $v_1$ not being the persistent centroid is at least $P_k/2$, which must in turn be less that $\epsilon$. This implies a bound on the required size of $k$. The value of $P_k$ and the corresponding bound on $k$ are developed in the following calculations.
\begin{enumerate}
\item \textbf{Preferential attachment:} We have
$$ P_k = \frac{1}{2k} \cdot \frac{2}{2k+2} \cdot \dots \cdot \frac{k-1}{4k-4} = \frac{(k-1)! (k-1)!}{2^{k-1} (2k-2)!} = \frac{1}{2^{k-1} {2k-2 \choose k-1}}.$$
Hence,
\begin{align*}
2\epsilon \geq P_k =  \frac{1}{2^{k-1} {2k-2 \choose k-1}} \geq \frac{1}{2^{k-1} 4^{k-1}} = \frac{1}{2^{3k-3}},
\end{align*}
using the fact that $\binom{2k-2}{k-1} \le 4^{k-1}$. Thus, a hub size of $k \ge C\log(1/ \epsilon)$ is necessary.
\item \textbf{Uniform attachment:} We have
$$P_k = \frac{1}{k+1} \cdot \frac{1}{k+2} \cdot \dots \cdot \frac{1}{2k-1} = \frac{k!}{(2k-1)!}.$$
Hence,
\begin{align*}
2\epsilon &\geq \frac{k!}{(2k-1)!} \geq \frac{k!}{(2k)!} \ge \frac{c\sqrt{k}(k/e)^k}{\sqrt{2k}(2k/e)^{2k}} = \frac{\tilde ce^k}{2^{2k}k^k},
\end{align*}
where $c$ and $\tilde c$ are suitable constants. Taking logarithms and simplifying, we obtain
$$\log(1/\epsilon) \leq k\log k + o(k \log k).$$
Thus, a hub size of $k \ge C'\frac{\log(1/\epsilon)}{\log \log (1/\epsilon)}$ is necessary.
\item \textbf{$d$-regular tree diffusion:} In a $d$-regular tree, we create an $r$-ball around vertex $v_1$ and derive bounds on the radius $r$ of the ball. Starting from the $r$-ball centered at $v_1$, we calculate the probability that vertices added in such a manner will make $v_1$ and $v_2$ symmetric and indistinguishable. To ensure this, the next $(d-1)^r$ vertices must be added to fill in the $r^{\text{th}}$ level in the subtree $(T, v_1)_{v_2\downarrow}$. This probability is equal to
$$P = \frac{(d-1)^r!}{\prod_{i=0}^{(d-1)^r-1} \Big(d(d-1)^r+ i(d-2)\Big)}.$$
Taking $\tau = \frac{d(d-1)^{r}}{d-2} \leq d^{r+1}$, we simplify this as
\begin{align*}
2\epsilon > P &= \frac{\left((d-1)^r\right)!}{(d-2)^{(d-1)^r}\prod_{i=0}^{(d-1)^r-1} \Big(\tau + i\Big)} \\
&\geq \frac{\left((d-1)^r\right)!}{(d-2)^{(d-1)^r}\prod_{i=0}^{(d-1)^r-1} \Big(d^{r+1} + i\Big)}\\
& = \frac{\left((d-1)^r\right)!(d^{r+1})!}{(d-2)^{(d-1)^r} \Big(d^{r+1}+ (d-1)^r - 1\Big)!}\\
&\ge \frac{1}{(d-2)^{(d-1)^r}\binom{d^{r+1}+ (d-1)^r}{(d-1)^r}}\\
&\geq \frac{1}{(d-2)^{(d-1)^r}2^{d^{r+1}+ (d-1)^r}}.
\end{align*}
Taking logarithms and simplifying, we then have
\begin{align*}
 (d-1)^r \log (d-2) + \Big(d^{r+1}+ (d-1)^r\Big) \log 2 \ge \log (1/2\epsilon).
\end{align*}
Since the left-hand side is $\Theta (d^{r+1})$, we obtain that a radius of size $r \ge C''\log \log (1/\epsilon)$ is necessary.
\end{enumerate}
\end{proof}

The next result provides sufficient conditions on the size of the initial hub ensuring persistence of the root node.

\begin{theorem*}\label{ThmHubSuff}
The following conditions are sufficient to ensure that $v_1$ is the persistent central node, where $\tilde{C}$ and $\tilde{C}'$ are appropriate constants:
\begin{itemize}
\item[(i)] For preferential and uniform attachment, the hub size $k$ satisfies $k \ge \tilde{C}\log (1/\epsilon)$.
\item[(ii)] For diffusion over a $d$-regular tree, with $d \ge 3$, the radius $r$ satisfies $r \ge \tilde{C}' \log \log (1/\epsilon)$.
\end{itemize}
\end{theorem*}

\begin{proof}
In the case of preferential or uniform attachment, suppose we start with a hub of size $K$, so $\{v_2, \dots, v_{K+1}\}$ are all connected to vertex $v_1$. Let $\cF$ be the event that $v_1$ becomes the persistent centroid, and let $\cF^*$ be the event that $v_1$ is a tree centroid at all time points. Clearly, $\mathbb P(\cF) \geq \mathbb P(\cF^*)$. We will select the hub size to ensure that the latter probability is at least $1- \epsilon$. Define the events
\begin{align*}
\hat H_i = \{\text{$v_i$ becomes a centroid at some time step}\}.
\end{align*}
Then
\begin{equation*}
\big(\cF^*\big)^c = \cup_{i=2}^\infty \hat H_i.
\end{equation*}
Note that
$$\hat H_i \subseteq H_i := \{\text{$v_i$ becomes at least as central as $v_1$ at some time step}\},$$
and
$$\hat H_i \subseteq G_i := \{\text{$v_i$ becomes at least as central as the centroid of $T_{i-1}$ at some time step}\}.$$
Thus,
\begin{align*}
\mathbb P\Big(\big(\cF^*\big)^c\Big) \leq \mathbb P \Big(\big(\cup_{i=2}^{K+1}  H_i \big) \bigcup \big(\cup_{i = K+2}^\infty  G_i \big)\Big).
\end{align*}
Using the bound from Lemma~\ref{lemma: galashin} in Appendix~\ref{AppRW}, and for $K$ greater than an appropriate constant, we then have
\begin{align}
\label{EqnCocoPuff}
 P\Big(\big(\cF^*\big)^c\Big) &\leq K \cdot \frac{P(K)}{2^K} + \sum_{i=K+2}^\infty  \frac{P(i/2)}{2^{i/2}} \notag \\
 &\stackrel{(a)}\leq \frac{2^{K/2}}{2^K} + \sum_{i=K+2}^\infty  \frac{2^{i/4}}{2^{i/2}} \notag \\
 &= 2^{-K/2} + \frac{2^{-(K+2)/4}}{1 - 2^{-1/4}} \notag \\
 &\leq 5\cdot 2^{-K/4},
\end{align}
where in $(a)$, we have used the fact that for large enough $K$, 
$$P(K ) < KP(K) < 2^{K/2}.$$
We can make $5\cdot 2^{-K/4} < \epsilon$ by choosing $K \ge \tilde{C}\log(1/\epsilon)$ for a large enough constant $\tilde{C}$.

For diffusion over a $d$-regular tree, the same calculation~\eqref{EqnCocoPuff} holds, except now
\begin{equation*}
K = \frac{(d-1)^{r+1}-1}{d-2} = \Theta(d^r)
\end{equation*}
for the initial seed graph. Thus, $r \ge \tilde{C}'\log \log (1/\epsilon)$ is a sufficient condition for the size of the radius.
\end{proof}

\begin{remark}
Comparing the necessary and sufficient conditions in Theorems~\ref{ThmHubNec} and~\ref{ThmHubSuff}, we see that a threshold occurs at hub size $k = \Theta\left(\log(1/\epsilon)\right)$ in the case of the preferential attachment model, and at radius \mbox{$r = \Theta\left(\log\log(1/\epsilon)\right)$} in the case of a $d$-regular diffusion. However, the bounds on the hub size in the case of uniform attachment disagree by a factor of $\log\log(1/\epsilon)$. It is a topic of future work to determine the exact threshold in this case, which must lie between $\Omega\left(\frac{\log(1/\epsilon)}{\log \log(1/\epsilon)}\right)$ and $O\left(\log(1/\epsilon)\right)$ by our results.
\end{remark}


\section{Discussion}
\label{SecDiscussion}

We have established the persistence of a unique centroid (or set of top $K$ central nodes) in three types of random growing trees: Uniform attachment, preferential attachment, and diffusion processes over $d$-regular trees. Furthermore, we have derived necessary and sufficient conditions for the size of the initial seed graph required to ensure that the first node is the persistent centroid in the network with probability $1-\epsilon$. A number of related open questions remain:
\begin{itemize}
\item[(i)] We believe that the results in this paper regarding persistence of the centroid should hold in more general preferential attachment models, where the probability of attaching to a node is proportional to a function $f$ of the vertex degree. In Galashin~\cite{Gal14}, it was shown that the degree-central node persists when $f$ is a convex function. Results concerning nonlinear P\'{o}lya urns indicate that for concave $f$, degree-centrality cannot persist~\cite{LarEtAl13}. However, the centroid persists when $f$ is either a linear or a constant function, and we conjecture that the persistence of the centroid holds for a larger class of functions, if not for all functions.

\item[(ii)] Our results and those from Galashin~\cite{Gal14} show that the top central node ``stabilizes" after a finite time, but we are unable to provide estimates on the expected time or the distribution of the time when stabilization occurs. This is particularly relevant for practical purposes, when one may wish to guarantee that the current centroid is the persistent centroid.

\item[(iii)] As mentioned in the remark after Theorem \ref{ThmHubSuff}, the problem of determining the hub-size threshold for the case of uniform attachment trees is still an open question. A related topic concerning degree centrality would be to provide necessary and sufficient conditions on the hub size in order to ensure degree centrality of $v_1$ (as in Section~\ref{SecHub}) in the convex preferential attachment model. It would be interesting to compare these conditions to the bounds required for the form of centrality studied in this paper.

\item[(iv)] In general, the degree-central node in a tree need not be the same as the centroid. For the preferential attachment model, one might ask whether the \emph{persistent} degree-central node is the same as the \emph{persistent} centroid with probability 1. It is tempting to think that such a result should hold; however, it probably does not. A heuristic argument is as follows: Consider a tree $T$ rooted at $v_1$ with neighbors $\{v_2, v_3, v_4,v_5\}$. Assume $v_2$ has a large number of children (say, $10^6$) and no grandchildren, and assume $(T, v_1)_{v_i \downarrow}$, for $i \in \{2,3,4\}$, is simply a line graph with, say, $10^{10}$ nodes. A preferential attachment process starting from such a tree would likely have $v_1$ as the persistent centroid, and $v_2$ as the persistent degree-central node. Since one can obtain $T$ with a finite probability starting from $v_1$, the persistent degree-central node cannot agree with the persistent centroid with probability 1. It would be interesting to study what additional constraints could ensure the agreement of both persistent nodes.

\item[(v)] Our results show that the top $K$ central nodes obtained according to the centrality measure $\psi$ stabilizes after a finite number of steps. However, a confidence set constructed according to $\psi$ may be sub-optimal in terms of the size of the set required as a function of the error probability $\epsilon$~\cite{BubEtal14, KhiLoh15}. It would be interesting to see whether other centrality measures such as those corresponding to the maximum likelihood estimator are also ``robust" in the sense that they produce a stable output after some finite time.

\item[(vi)] The problem of establishing persistence of centrality in non-trees (for example, in preferential or uniform attachment models where more than one node is added at each step) appears to be very challenging. It is not clear what notion of centrality, if any, would persist in such cases. Even for trees, the problems of establishing persistent centrality in alternative models such as preferential or uniform attachment with choice~\cite{Mal14, Has14}, or random tree branching processes, are worth  considering.

\end{itemize}

\section*{Acknowledgments}

The authors thank Elchanan Mossel for suggesting the counterexample to the conjecture that the persistent degree-central node agrees with the persistent centroid in the general preferential attachment model.

\appendix

\section{Weighted 2-dimensional random walks}
\label{AppRW}

In this section, we consider a class of random walks on $\mathbb N \times \mathbb N$. From position $(i,j)$, at each time step the random walk can move either  one step up with probability $U(i,j)$, or one step to the right with probability $R(i,j)$. The probabilities of these movements depend on $(i,j)$ according to
\begin{equation*}
R(i, j) \propto \alpha i + \beta, \quad \text{ and } \quad U(i,j) \propto \alpha j + \beta,
\end{equation*}
for some $\alpha > 0$ and $\alpha + \beta \geq 0$.
Our next lemma pertains to such random walks:

\begin{lemma}\label{lemma: galashin}
Consider a 2-dimensional random walk on the $\mathbb N \times \mathbb N$ lattice, where the location of the walk at time $n$ is denoted by $W_n$, and the probabilities of movement are given by
\begin{align*}
\mathbb P \big(W_{n+1} = (i+1, j) | W_n = (i,j)\big) &= R(i,j) \propto \alpha i + \beta, \quad \text{and} \\
\mathbb P\big(W_{n+1} = (i, j+1) | W_n = (i,j)\big) &= U(i,j) \propto \alpha j + \beta.
\end{align*}
For $A > 2$, let $f(A)$ be the probability that the random walk reaches the diagonal at some future time when it starts at $W_0 = (A, 1)$. Then there exists a fixed polynomial $P$ such that
\begin{equation*}
f(A) \leq \frac{P(A)}{2^A},
\end{equation*}
and $\frac{P(A)}{2^A}$ is monotonically decreasing for sufficiently large $A$.
\end{lemma}

\begin{proof}
Let $f(A,m)$ be the probability that the random walk lies entirely below the diagonal before reaching $(m,m)$ on the diagonal. Clearly,
\begin{equation*}
f(A) = \sum_{m = A}^\infty f(A, m).
\end{equation*}
We will now bound $f(A,m)$ for $m \geq A$. Let $\Theta((A,B) \to (m,m))$ be the number of paths from $(A, B)$ to $(m,m)$, such that every point on the path lies strictly below the diagonal, except for the endpoint $(m,m)$. Using the reflection principle, Lemma 2 of Galashin~\cite{Gal14}  shows that the total number of such paths from $(A, B)$ to $(m,m)$ is given by the expression
\begin{equation*}
\Theta((A,B) \to (m,m)) = \frac{(2m-1-A-B)!(A-B)}{(m-A)!(m-B)!}.
\end{equation*}
Substituting $B=1$, we have
$$\Theta((A,1) \to (m,m)) = \frac{(2m-2-A)!(A-1)}{(m-A)!(m-1)!} =\frac{\Gamma(2m-A-1)(A-1)}{\Gamma(m+1-A)\Gamma(m)}.$$
Furthermore, every path from $(A, B)$ to $(m,m)$ has the same probability $p((A,B) \to (m,m))$, and
\begin{align*}
p((A,B) \to (m,m)) &= \frac{\prod_{i = A}^{m-1}(\alpha i + \beta) \prod_{j=B}^{m-1} (\alpha j + \beta)}{\prod_{k = A+B}^{2m-1} (\alpha k + 2\beta)}\\
&= \frac{\alpha^{m-A} \prod_{i = A}^{m-1}(i + \beta/\alpha) \times \alpha^{m-B}\prod_{j=B}^{m-1} (j + \beta/\alpha)}{\alpha^{2m-A-B} \prod_{k = A+B}^{2m-1} (k + 2\beta/\alpha)}\\
&= \frac{\prod_{i = A}^{m-1}(i + \beta/\alpha) \prod_{j=B}^{m-1} (j + \beta/\alpha)}{\prod_{k = A+B}^{2m-1} (k + 2\beta/\alpha)}.
\end{align*}
 
Substituting $B=1$ gives
\begin{align*}
p((A,1) \to (m,m)) &=  \frac{\prod_{i = A}^{m-1}(i + \beta/\alpha) \prod_{j=1}^{m-1} (j + \beta/\alpha)}{\prod_{k = A+1}^{2m-1} (k + 2\beta/\alpha)}\\
&= \frac{\frac{\Gamma(m+\beta/\alpha)}{\Gamma(A+\beta/\alpha)} \frac{\Gamma(m+\beta/\alpha)}{\Gamma(1+\beta/\alpha)}}{\frac{\Gamma(2m + 2\beta/\alpha)}{\Gamma(A+1+2\beta/\alpha)}}\\
&= \frac{\Gamma(A+1+2\beta/\alpha)}{\Gamma(A+\beta/\alpha)\Gamma(1+\beta/\alpha)} \cdot \frac{\Gamma(m+\beta/\alpha)\Gamma(m+\beta/\alpha)}{\Gamma(2m + 2\beta/\alpha)}\\
& \le \frac{\Gamma(A + 1 + 2 \lceil\beta/\alpha\rceil)}{\Gamma(A + \lceil \beta/\alpha\rceil - 1) \Gamma(1+\beta/\alpha)} \cdot \frac{\Gamma(m+\beta/\alpha)\Gamma(m+\beta/\alpha)}{\Gamma(2m + 2\beta/\alpha)}\\
&\leq P(A) \cdot \frac{\Gamma^2(m+\beta/\alpha)}{\Gamma(2m + 2\beta/\alpha)},
\end{align*}
where $P(A)$ is a fixed polynomial with degree at most $\lceil \beta/\alpha \rceil + 2$. In the rest of this proof, we continue using $P(A)$ to denote a fixed polynomial, although its exact expression may change from line to line. Since
$$f(A, m) = \Theta((A,1) \to (m,m)) \cdot p((A,1) \to (m,m)),$$
we have the bound
\begin{align*}
f(A,m) &\leq P(A) \frac{\Gamma(2m-A-1)}{\Gamma(m+1-A)\Gamma(m)} \frac{\Gamma^2(m+\beta/\alpha)}{\Gamma(2m + 2\beta/\alpha)},
\end{align*}
for some polynomial $P(A)$.
We now use Stirling's bound, which says that for all $z > 0$, the value of $\Gamma(z)$ lies with a constant factor of $\frac{1}{\sqrt z} \left(\frac{z}{e}\right)^z$:
$$\Gamma(z) \sim \frac{1}{\sqrt z} \left(\frac{z}{e}\right)^z.$$
Substituting and modifying $P(A)$ as convenient, we then have
\begin{align*}
f(A,m) &\leq P(A) \frac{(2m-A-1)^{2m-A-1-1/2} (m+\beta/\alpha)^{2(m+\beta/\alpha-1/2)}}{(m+1-A)^{m+1-A-1/2} m^{m-1/2} (2m + 2\beta/\alpha)^{2m + 2\beta/\alpha -1/2}}\\
&= P(A) \frac{(2m)^{2m-A-3/2}m^{2m+2\beta/\alpha-1}}{m^{m-A+1/2} m^{m-1/2}(2m)^{2m+2\beta/\alpha-1/2}} \times \frac{(1 - \frac{A+1}{2m})^{2m-A-3/2} (1+ \beta/m\alpha)^{2m+2\beta/\alpha-1}}{(1- \frac{A-1}{m})^{m-A+1/2} (1+\beta/m\alpha)^{2m+2\beta/\alpha-1/2}  }\\
&= \frac{P(A)}{2^A m^2} \times \frac{(1 - \frac{A+1}{2m})^{2m-A-3/2}}{(1- \frac{A-1}{m})^{m-A+1/2} (1+\beta/m\alpha)^{1/2}  }\\
&\leq  \frac{P(A)}{2^A m^2} \times \frac{(1 - \frac{A+1}{2m})^{2m-A-3/2}}{(1- \frac{A-1}{m})^{m-A+1/2}}\\
&=  \frac{P(A)}{2^A m^2} \times \left(\frac{(1 - \frac{A+1}{2m})}{(1- \frac{A-1}{m})} \right)^{m-A+1/2} \left(1 - \frac{A+1}{2m}\right)^{m-2}\\
&=  \frac{P(A)}{2^A m^2} \times \left( 1 + \frac{A-3}{2m-2A+2} \right)^{m-A+1/2}\left(1 - \frac{A+1}{2m}\right)^{m-2}\\
&\stackrel{(a)}\leq  \frac{P(A)}{2^A m^2} \times \exp \left(  \frac{A-3}{(2m-2A+2)}(m-A+1/2) - \frac{(A+1)}{2m} (m-2)\right)\\
&\leq \frac{P(A)}{2^A m^2} \times \exp \left( \frac{A-3}{2} - \frac{A+1}{2} + \frac{A+1}{m}\right)\\
&\leq  \frac{P(A)}{2^A m^2} \times \exp \left( -2+2\right) \\
&= \frac{P(A)}{2^A m^2},
\end{align*}
where in $(a)$, we have used the fact that for all $x\in \mathbb R$, we have $1+x \leq e^x$. Finally, noting that $\sum_{m=1}^\infty \frac{1}{m^2} < \infty$, we conclude that
\begin{equation*}
f(A) = \sum_{m=A}^\infty f(A, m) \leq \frac{P(A)}{2^A},
\end{equation*}
for a fixed polynomial $P$. Without loss of generality, we may choose $P(A)$ to be a monomial with a positive coefficient, so $\frac{P(A)}{2^A}$ is clearly monotonically decreasing for large enough $A$.
\end{proof}

\bibliography{refs}

\end{document}